\documentclass[11pt,twoside]{amsart}

\usepackage{pstricks}
\usepackage{multido}

\usepackage{graphicx,url}
\usepackage{amssymb}
\usepackage{amsmath}
\usepackage{amscd}
\usepackage{amstext}
\usepackage{amsbsy}
\usepackage{amsxtra}
\usepackage{amsfonts}
\usepackage{latexsym}
\usepackage{alltt}
\usepackage{bbm}

\usepackage{verbatim}
\usepackage{amscd}
\usepackage[all]{xy}

\theoremstyle{plain}

\newtheorem{lemma}{Lemma}[section]
\newtheorem{corollary}[lemma]{Corollary}
\newtheorem{proposition}[lemma]{Proposition}

\newtheorem{definition}[lemma]{Definition}

\newcommand{\imm}{\mathrm{f}}
\newcommand{\cmc}{{\sc{cmc~}}}

\newcommand{\R}{\mathbb R}

\newcommand{\Sp}{\mathbb S}

\newcommand{\N}{\mathbb N}

\newcommand{\Order}{{\rm O}}

\newcommand{\lbcl}{c}
\newcommand{\ubdh}{C}
\newcommand{\cab}{C_0}
\newcommand{\Cab}{C_1}

\newcommand{\II}{\mathfrak{h}}
\newcommand{\normal}{\mathfrak{N}}

\newcommand{\vn}{}

\DeclareMathOperator{\arccot}{arccot}
\DeclareMathOperator{\dist}{dist}

\newcommand{\ind}[1]{_{\mbox{\rm\scriptsize{#1}}}}

\numberwithin{equation}{section}

\title[Mean-convex Alexandrov embedded surfaces in $\Sp^3$]
{On mean-convex Alexandrov embedded surfaces in the 3-sphere}

\author{L. Hauswirth}
\address{\tiny L. Hauswirth, Universit\'e de Marne la Vall\'ee, France.}
\email{\tiny hauswirth@univ-mlv.fr}
\author{M. Kilian}
\address{\tiny M. Kilian, University College Cork, Ireland.}
\email{\tiny m.kilian@ucc.ie}
\author{M. U. Schmidt}
\address{\tiny M. U. Schmidt, Universit\"at Mannheim, Germany.}
\email{\tiny schmidt@math.uni-mannheim.de}

\begin{document}



\begin{abstract}
We consider mean-convex Alexandrov embedded surfaces in the round unit 3-sphere, and show under which conditions it is possible to continuously deform these preserving mean-convex Alexandrov embeddedness.
\end{abstract}
\maketitle


\section{Introduction}
In this paper we consider a class of complete immersed surfaces $M \hookrightarrow \Sp^3$ in the round unit 3-sphere $\Sp^3$. The additional property we assume is that these surfaces are mean-convex Alexandrov embedded. This condition means that the immersion $M \hookrightarrow \Sp^3$ extends to an immersion $N \hookrightarrow \Sp^3$ of a 3-manifold $N$, whose boundary $\partial N = M$, and such that $N$ lies on the mean-convex side of $M$. In the literature there are the notions of Alexandrov embeddings for compact domains on the one hand, and the concept of properly Alexandrov embedded immersions from open manifolds into open Riemannian manifolds on the other hand. Since we are interested in immersions of open manifolds into the compact Riemannian manifold $\Sp^3$, we propose the following
\begin{definition}\label{mean-convex}
A mean-convex Alexandrov embedded surface in $\Sp^3$ is a smooth complete immersion $\imm:M \to \Sp^3$ from a connected surface $M$ which extends as an immersion to a connected 3-manifold $N$
with boundary $M=\partial N$ with the following properties:

{\rm{(i)}} The mean curvature of $M$ in $\Sp^3$ with respect to the
  inward normal is non-negative everywhere.

{\rm{(ii)}} The manifold $N$ is complete with respect to the metric
  induced by $\imm$.

An immersion $\imm:M\to\Sp^3$ just obeying condition~{\rm{(ii)}} is called
an Alexandrov embedding.
\end{definition}
Preserving the property of mean-convex Alexandrov embeddedness during continuous deformations is subtle when the surface is not compact. In general one gets continuity only with respect to the topology defined by uniform distances on compact subsets, and this not necessarily preserves embeddness or Alexandrov embeddness of non-compact surfaces. We shall see that a mean-convex Alexandrov embedded surface permits a deformation preserving Alexandrov embeddedness, if firstly the surface bounds a collar of width bounded uniformly from below and if secondly the surface obeys a chord-arc-bound. Furthermore, we show that for a parabolic surface $M$ any mean-convex Alexandrov embedding $\imm : M \to \Sp^3$ with constant mean curvature fulfills these two conditions, if it has {\emph{uniformly bounded geometry}}. By this we mean:

{\rm{(a)}} The principal curvatures $\kappa_1,\,\kappa_2$ of $\imm$ are uniformly bounded by  \begin{equation} \label{eq:bounded_curvature}
        \max(|\kappa_1|, |\kappa_2|) \leq \kappa_{\max}\,.
    \end{equation}
{\rm{(b)}} There is a constant $\ubdh$ which bounds the covariant derivative of the second fundamental form $\II$ of the immersion, so that for all $p\in M$ and $X,\,Y,\,Z \in T_p M$ we have
\begin{equation}\label{eq:upper bound derivative of h}
|(\nabla_X \II)(Y,Z)| \leq \ubdh\,|X|\,|Y|\,|Z|\,.
\end{equation}
Condition (a) is sufficient to guarantee the existence of a global positive lower bound of the cut-locus function for \cmc mean-convex Alexandrov embedded parabolic surfaces (Proposition \ref{th:Rosenberg_Lemma}). This uniform lower bound $c>0$ of the cut-locus function is obtained with Hopf's maximum principle at infinity, and in turn provides us with a collar of uniform width. Condition (b) and the lower bound on the cut-locus are needed to prove a chord-arc bound (Proposition~\ref{thm:ec diameter bound}), which relates distances in $M$ to distances in $N$.
In order to preserve mean-convex Alexandrov embeddedness during continuous deformations, we localize the concept, and divide the surface into compact mean-convex Alexandrov embedded pieces. We show first that along continuous $\mathrm{C}^1$-deformations the compact pieces stay mean-convex Alexandrov embedded (Proposition~\ref{thm:collar deformation 1}). We then prove that these local mean-convex Alexandrov embedded pieces can be glued together into a global mean-convex Alexandrov embedding (Propsition~\ref{thm:collar deformation 2}). We combine our results into the following

{\bf Theorem.} {\it
Let $\imm :M\to\Sp^3$ be an immersion with non-negative mean curvature and principal curvatures bounded by $\kappa\ind{max}=\cot(\lbcl)$. Suppose that for each $p\in M$ a mean-convex Alexandrov embedding $\imm_p : M\to\Sp^3$ with cut-locus function bounded from below by $\lbcl>0$ and second fundamental form obeying \eqref{eq:upper bound derivative of h} with $\ubdh>0$ is on $B(p,R)\subset M$ close to $\imm$ in the sense that
\begin{align*}
    \|\,\imm-\imm_p \,\|_{\mathrm{C}^1(B(p,R))}&<\epsilon
\end{align*}
for $\epsilon,R>0$ depending on $\lbcl$ and $\ubdh$. Then $\imm$ extends to a mean-convex Alexandrov embedding.}

As an application of our result we note that surfaces in $\Sp^3$ with constant mean curvature $H$ are not isolated, but come in deformation families with varying $H$. These deformations come from the integration of bounded Jacobi fields on the surface. We consider not only one immersed \cmc surface $\imm: M \to \Sp^3$ but a whole deformation family $\{\imm_t \}$. For bounded Jacobi fields one can show that if one member of a family is mean-convex Alexandrov embedded with cut-locus function bounded from below by $c>0$, then every member of the family is mean-convex Alexandrov embedded, and the cut locus functions of the whole family is uniformly bounded from below by $c>0$. When deforming a mean-convex Alexandrov embedded surface $\imm$ of uniformly bounded geometry by integration of a non-bounded Jacobi field, then we need to control the geometry at infinity of surfaces $\imm_t$ which are close to $\imm$ on compact subsets. Our Theorem ensures that if the geometry at infinity of $\imm_t$ is uniformly close to $\imm$ on all balls of fixed radius, then $\imm_t$ is also mean-convex Alexandrov embedded. Such an application might be considered in future work.

\section{Lower bound on the cut locus function}
\label{sec:AE space}
\subsection{Inward $M$-geodesics} In the setting of Definition~\ref{mean-convex}, a fixed orientation of $\Sp^3$ induces on $N$ and $M=\partial N$ an orientation. Conversely, if $M$ is endowed with an orientation, then there exists a unique normal, which points inward to the side of $M$ in $\Sp^3$, which induces on the boundary $M$ the given orientation of $M$. In this sense the orientation of $M$ determines the inward normal of $N$. For each point $p \in M$ of a hypersurface of a Riemannian manifold $N$ there exists a unique arc-length parameterised geodesic $\gamma(p,\cdot)$ emanating from $p = \gamma(p,0)$ and going in the direction of the inward normal at $p$. Such geodesics are called inward $M$-geodesics.

\subsection{Cut locus function} Let $\gamma(p,\cdot)$ be an inward $M$-geodesic. Points $q \in N$ in the ambient manifold that are `close to one side' of $M$ can thus be uniquely parameterised by $(p,\,t)$ where $p \in M$ and $q=\gamma(p,t)$ for some inward $M$-geodesic $\gamma(p,\cdot)$ and some $t \in \R_0^+$. The value of $t$ is the geodesic distance of $q$ to $M$. Extending the geodesic further into $N$ it might eventually encounter a point past which $\gamma(p,t)$ has distance to $M$ smaller than $t$. Such a point is called a cut point. The cut locus of $M$ in $N$ consists of the set of cut points along all inward $M$-geodesics. We define the cut locus function as the geodesic distance of the cut point to $M$:
\begin{equation}\label{eq:cutlocus}
  c:M\to\R^+,\quad p\mapsto c(p),\quad\mbox{ such that }
\gamma(p,c(p))\mbox{ is the cut point.}
\end{equation}
If we want to emphasise the dependence on $\imm$ we write $\gamma_\imm$ and $c_\imm$.

\subsection{First focal point} Recall the following characterization of cut points (see e.g. \cite[Lemma~2.1]{Heb}).
\begin{lemma}\label{th:cut-point}  A cut point is either the first focal point on an inward $M$-geodesic, or it is the intersection point of two shortest inward $M$-geodesics of equal length.
\end{lemma}
The first focal point on an inward $M$-geodesic is denoted by
\[
    \gamma(p,t\ind{foc})\,.
\]
For mean-convex $M$ the first focal point on an inward $M$-geodesic has geodesic distance at most $\tfrac{\pi}{2}$, so that $t\ind{foc} \leq \tfrac{\pi}{2}$. Hence also the cut locus function of a mean-convex Alexandrov embedding is uniformly bounded from above by $\frac{\pi}{2}$, since $c \leq t\ind{foc}$. This yields the following
\begin{corollary} The cut-locus function of a mean-convex Alexandrov embedded surface in $\Sp^3$ is uniformly bounded above by $\tfrac{\pi}{2}$.
\end{corollary}
\subsection{Generalized cylinder coordinates}
For a mean-convex Alexandrov embedded surface $\imm:M \to \Sp^3$, the inward $M$-geodesics give us a parametrisation of the 3-manifold $N$ with boundary $M = \partial N$, which we call generalised cylinder coordinates:
\begin{equation}\label{eq:cylinder}
  \gamma_{\imm} : \{(p,t)\in M\times\R\mid 0 \leq t < c_{\imm}(p)\}
  \to N\,.
\end{equation}
These coordinates define a diffeomorphism onto the complement of the cut locus. The cut locus is homeomorphic to the quotient space $M/\sim_{\imm}$ with the following equivalence relation on $M$:
\begin{align*}
    p&\sim_{\imm} q&&\Longleftrightarrow&
    c_{\imm}(p)&=c_{\imm}(q)\quad\mbox{ and }&
    \gamma_{\imm}(p,c_{\imm}(p))&=\gamma_{\imm}(q,c_{\imm}(q))\\
    &&&\Longleftrightarrow&&&\gamma_{\imm}(p,c_{\imm}(p))&=
    \gamma_{\imm}(q,c_{\imm}(q)).
\end{align*}
For each $p\in M$ we denote the corresponding equivalence class by
\begin{align}\label{eq:ec cutlocus}
    [p]_{\imm}&=\{q\in M\mid \gamma_{\imm}(p,c_{\imm}(p))=
    \gamma_{\imm}(q,c_{\imm}(q))\}.
\end{align}
\begin{proposition}\label{th:Rosenberg_Lemma}
Let $\imm:M \to \Sp^3$ be a mean-convex Alexandrov embedding with constant mean curvature and principal curvatures bounded by $\kappa\ind{max}>0$ and assume that the immersion is conformally parabolic. Then the cut locus function is bounded from below by $\arctan(\kappa\ind{max}^{-1})$.
\end{proposition}
\begin{proof}
For the hypersurface $M_t = \cup_{p\in M} \gamma(p,t)$, the mean curvature
\begin{equation}\label{eq:mean curvature}
  H(t) = \tfrac{1}{2}\left(
  \cot\left(\arctan(\kappa_1^{-1})-t\right) +
  \cot\left(\arctan(\kappa_2^{-1})-t\right) \right)
\end{equation}
is positive for all $t \in (0,\,t\ind{foc})$, and strictly increasing since
$$
    H'(t) = \tfrac{1}{2}\left(\sin^{-2}(\arctan(\kappa_1^{-1})-t) +
    \sin^{-2}(\arctan(\kappa_2^{-1})-t)\right) > 0\,.
$$
Let $c_{\imm}$ denote the cut locus function \eqref{eq:cutlocus}. If there exists a point $p \in M$ for which $c_{\imm}(p) < \arctan(\kappa\ind{max}^{-1})\leq t\ind{foc}$, then two inward $M$-geodesics $\gamma(p,\cdot),\,\gamma(q,\cdot)$ through $p,\,q \in M$ respectively, have to intersect at a distance of $c_{\imm}(p)$ from $M$, and thus $c_{\imm}(p) = c_{\imm}(q)$. Hence, if there exists a point $p \in M$ with $c_{\imm}(p) <\arctan(\kappa\ind{max}^{-1})$ then $M_t$ intersects itself for a value of $t < \arctan(\kappa\ind{max}^{-1})$ over two points $p,\,q \in M$. Let
$$
    \lbcl_0 = \inf\{t\,\left| \right. M_t \mbox{ intersects over two
    points of } M\}.
$$
Since over all points $p \in M$ the mean curvature of $M_t$ is positive for all $0<t<\arctan(\kappa\ind{max}^{-1})$ with respect to the inner normals, the surfaces $M_t$ cannot intersect themselves with opposite sign of the mean curvature vector over two points of $M$ for $0<t<\arctan(\kappa\ind{max}^{-1})$. This implies that the cut locus function has no local minima less than $\arctan(\kappa\ind{max}^{-1})$.

Now let $(p_k)_{k\in\N}$ be a sequence in $M$ with
$$
    \lim_{k \to\infty}c_{\imm}(p_k) =\lbcl_0 = \inf\left\{c_{\imm}(p) \mid p \in
    M \right\}\,.
$$
Then there exists a sequence $\Theta_k$ of isometries of $\Sp^3$ which transform each point $p_k$ into a fixed reference point $p_0\in\Sp^3$, and the tangent plane of $M$ at $p_k$ into the tangent plane of a fixed geodesic sphere $\Sp^2_{p_0}\subset \Sp^3$ which contains $p_0$. The immersion $\imm$ is locally a normal graph over balls in $\Sp^2$ of radius $r>0$ depending only on $\kappa\ind{max}$. Therefore this sequence of isometries transforms neighbourhoods $U_k$ of $p_k\in M$ into normal {\sc{cmc}} graphs $\Theta_k[U_k]$ over $B(p_0,r)\subset\Sp^2_{p_0}$. Due to Arzel\`{a}-Ascoli, and the a-priori gradient bound from Proposition~4.1 in \cite{ForLR}, this bounded sequence of normal {\sc{cmc}} graphs over $B(p_0,r) \subset \Sp^2_{p_0}$ has a convergent subsequence. By passing to a subsequence we may achieve that these graphs converge to a normal {\sc{cmc}} graph $U$ over $B(p_0,r) \subset \Sp^2_{p_0}$, which is tangent to $\Sp^2_{p_0}$ at $p_0$. For $\lbcl_0<\arctan(\kappa\ind{max}^{-1})$ the sets $[p_k]_{\imm}$ contain besides $p_k$ another point $q_k$ for large $k$. Furthermore, the sequence of isometries $\Theta_k$ transforms the sequence of geodesic 2-spheres tangent to $M$ at $q_k$ into a converging sequence of spheres with limit $\Sp^2_{q_0}$. This sphere contains the limit $q_0=\lim\Theta_k(q_k)$ with distance $\mathrm{dist}(p_0,q_0)\le 2\lbcl_0$. For large $k$ the points $q_k$ have neighbourhoods $V_k$ such that $\Theta_k[V_k]$ are normal {\sc{cmc}} graphs over $B(q_0,r)\subset\Sp^2_{q_0}$. By passing again to a subsequence the normal {\sc{cmc}} graphs $\Theta_k[V_k]$ converge to a normal {\sc{cmc}} graph $V$ tangent to $\Sp^2_{q_0}$ at $q_0$. The transformed inward $M$-geodesics nearby $p_k$ and $q_k$ converge to normal geodesics of these two limiting {\sc{cmc}} surfaces $U$ and $V$ in $\Sp^3$. Let $U_{\lbcl_0}$ and $V_{\lbcl_0}$ denote the surfaces $U$ and $V$ shifted by $\lbcl_0$ along these normal geodesics. If we shift both sequences $\Theta_k[U_k]$ and $\Theta_k[V_k]$ by $\lbcl_0$ along the transformed $M$-geodesics, they converge to $U_{\lbcl_0}$ and $V_{\lbcl_0}$. Therefore these surfaces $U_{\lbcl_0}$ and $V_{\lbcl_0}$ touch each other at a first point of contact with opposite sign of the mean curvature at the limit of the transformed cut points $\lim\Theta_k(\gamma_{\imm}(p_k,c(p_k)))= \lim\Theta_k(\gamma_{\imm}(q_k,c(q_k)))$, contradicting Hopf's maximum principle. Hence the shifted surfaces cannot have positive mean curvature with respect to the inner normal and this implies $\lbcl_0=0$ and $H=0$. This however contradicts the maximum principle at infinity \cite[Theorem~7]{Maz} as follows: The generalised cylinder coordinates~\eqref{eq:cylinder} define an immersion from the 3-manifold $(x,t)\in M\times[0,\arctan(\kappa_{\max}^{-1})]$ with boundary components $M\times\{0\}$ and $M\times\{\arctan(\kappa_{\max}^{-1})\}$. Each surface $M \times \{t\}$ has uniform bounded curvature, parabolic conformal type and are mean-convex with mean curvature $H>0$ pointing away from the minimal surface $M \times \{0\}$. If we assume the infimum $c>0$ of the cut locus function smaller than $\arctan(\kappa_{\max}^{-1})$, some pieces of $M \times \{c\}$ sit inside this 3-manifold with possible boundaries in the second boundary component $M\times\{\arctan(\kappa_{\max}^{-1})\}$. All assumptions of \cite[Theorem~7]{Maz} are fulfilled, and these pieces contradict the maximum principle since the mean curvature vector points into $M \times \{0\}$, giving the contradiction.
\end{proof}
\section{Chord-arc bound}\label{sec:chord arc}
Let $\imm:M\to\Sp^3$ be a mean-convex Alexandrov embedding with cut locus function $c_{\imm}(p)$ \eqref{eq:cutlocus} bounded from below by $\lbcl>0$. The generalized cylinder coordinates \eqref{eq:cylinder} define a diffeomorphism $\gamma_{\imm}$ of $M\times[0,\lbcl)$ onto an open subset of $N$, which is a collar. Due to Lemma \ref{th:cut-point} for all $p\in M$ the first focal point $\gamma(p,t\ind{foc})$ has distance $t\ind{foc}=\arctan\left((\max\{\kappa_1,\,\kappa_2\})^{-1}\right)\ge c_\imm(p)\ge\lbcl$ to $M$. Hence both principal curvatures are uniformly bounded by $\kappa\ind{max}=\cot(\lbcl)$. Consequently, due to the formula~\eqref{eq:mean curvature}, the distances to the first focal points on the outward $M$-geodesics are not smaller than the distances to the first focal points on the inward $M$-geodesic which are at least $\lbcl$. Hence the normal variation defines an immersion of $M\times(-\lbcl,\lbcl)$ into $\Sp^3$. In particular, the induced metric makes $M\times[-\frac{\lbcl}{2},\lbcl)$ into a Riemannian manifold with constant sectional curvature equal to one and with boundary $M\times\{-\frac{\lbcl}{2}\}$. For all elements of this manifold the cylinder coordinates, that is the distances to $M\simeq M\times\{0\}$ and the nearest point in $M$ are uniquely defined. Hence we can glue $M\times[-\frac{\lbcl}{2},\lbcl)$ along $\gamma_{\imm}(M\times[0,\lbcl))$ to $N$, and obtain a larger complete 3-manifold $\widehat{N}\supset N$ with boundary $\partial\widehat{N}=M\times\{-\frac{\lbcl}{2}\}$. The generalised cylinder coordinates \eqref{eq:cylinder} extend to an embedding $\hat{\gamma}_{\imm}:M\times[-\frac{\lbcl}{2},\lbcl)\hookrightarrow\widehat{N}$ and $\imm$ extends to an immersion $\hat{\imm}:\widehat{N}\to\Sp^3$.

Let $\dist_M$, $\dist_N$ and $\dist_{\widehat{N}}$ denote the distance functions of the complete Riemannian manifolds $M\subset N\subset\widehat{N}$. For $p,q\in M$ we have the obvious inequalities
\[
    \dist_{\widehat{N}}(p,q)\le\dist_N(p,q)\le\dist_M(p,q)\,.
\]
Therefore the following Proposition implies the chord-arc bound
\begin{equation}\label{eq:chord arc 0}
    \dist_N(p,q)\leq\dist_M(p,q)\le\cab\dist_{\widehat{N}}(p,q)\leq\cab\dist_N(p,q)\quad
    \mbox{ for all }p,q\in M.
\end{equation}
\begin{proposition}\label{thm:ec diameter bound}
Let $\imm:M\to\Sp^3$ be a mean-convex Alexandrov embedding with second fundamental form $\II$ with respect to the inner normal $\normal$. If $\lbcl>0$ is a lower bound on the cut locus function $c_{\imm} \geq c>0$ \eqref{eq:cutlocus} and $\ubdh$ a bound on the covariant derivative of $\II$:
\begin{align*}
|(\nabla_X\II)(Y,Z)|&
\leq\ubdh\cdot|X|\cdot|Y|\cdot|Z|&
\text{for all } p\in M&\text{ and }X,Y,Z\in T_pM,
\end{align*}
then there exists a constant $\cab>0$ depending only on $\lbcl$ and $\ubdh$ such that
\begin{equation}\label{eq:chord arc}
    \dist_{\widehat{N}}(p,q)\leq\dist_M(p,q)\leq\cab\dist_{\widehat{N}}(p,q)\quad
    \mbox{ for all }p,q\in M.
\end{equation}
\end{proposition}
\begin{proof}
For all $p,q\in M$ we have $\dist_{\widehat{N}}(p,q)\leq\dist_M(p,q)$. We prove the reverse inequality by constructing a path from $p$ to $q$ in $M$ of length at most $\cab\dist_{\widehat{N}}(p,q)$. Due to (Rinow~\cite{Rin}, pages 172 and 141) the points $p$ and $q$ are joined by a shortest path in $\widehat{N}$. If this shortest path passes through $M\subset\widehat{N}$, then we may divide the shortest into several segments and prove~\eqref{eq:chord arc} for each segment separately. We distinguish between the following cases:
\begin{enumerate}
\item[\bf{(A)}]\label{group 1} The shortest path stays inside $\{\hat{\gamma}_{\imm}(p,t)\mid(p,t)\in M\times[-\frac{\lbcl}{2},\frac{\lbcl}{2}]\}$.
\item[\bf{(B)}] The shortest path stays inside $N$ with maximal distance $>\frac{\lbcl}{2}$ to $M$.
\end{enumerate}
In case {\bf{(A)}} the first entries of the cylinder coordinates of the shortest path yields a path in $M$ from $p$ to $q$. We estimate the length of the derivative $d\hat{\gamma}_{\imm}(p,t)(p',t')$ of the generalised cylinder coordinates at $(p,t)\in M\times[-\frac{\lbcl}{2}, \frac{\lbcl}{2}]$ in direction of $(p',t')\in T_{(p,t)}M\times(-\lbcl,\lbcl)$ from below. For this purpose we decompose this derivative in components parallel and orthogonal to the $M$-geodesic. The length of component orthogonal to the $M$-geodesic is a lower bound of the length of the derivative and does not depend on $t'$. The vector field on $\Sp^2$ of rotations around the $z$-axis has length proportional to $\sin$ of the geodesic distance to the poles. Therefore on a geodesic sphere of radius $\lbcl=\arccot(\kappa\ind{max})$ the length of the orthogonal component is equal to
$$|d\hat{\gamma}_{\imm}(p,t)(p',0)|=\tfrac{\sin(\lbcl-t)}{\sin(\lbcl)}|p'|
=(\cos(t)-\cot(\lbcl)\sin(t))|p'|=(\cos(t)-\kappa\ind{max}\sin(t))|p'|.
$$
If both principal curvatures are smaller than $\kappa\ind{max}$, then the orthogonal component is longer. Therefore for $t\le\frac{\lbcl}{2}<\frac{\pi}{2}$ we obtain the lower bound
\begin{align*}
|d\hat{\gamma}_{\imm}(p,t)(p',t')|
\ge\left(\cos(\tfrac{\lbcl}{2})-\cot(\lbcl)\sin(\tfrac{\lbcl}{2})\right)|p'|
=\frac{1}{2\cos(\tfrac{\lbcl}{2})}|p'|&\ge\tfrac{1}{2}|p'|.
\end{align*}
The integral of this inequality along the shortest path yields for all such pairs $(p,q)$ in case~{\bf(A)}
\begin{align}\label{eq:chord arc 1}
\dist_M(p,q)&\le 2\dist_{\widehat{N}}(p,q).
\end{align}
This includes all $(p,q)$ with $\dist_{\widehat{N}}(p,q)<\lbcl$.

In case~{\bf(B)} the shortest path is a geodesic of $N$ with only two boundary points at both ends.

\noindent{\it Claim.} Suppose $p,\,q \in M$ are connected in $N$ by a geodesic of length $\geq\pi$. Then
$$
   \dist_N(p,\tilde{q}) < \pi\quad\mbox{ and }\quad
   \dist_N(p,q)=\dist_N(p,\tilde{q})+\dist_N(\tilde{q},q)
   \quad\mbox{ for some }\quad\tilde{q}\in M.
$$
To prove this claim, note that all geodesics $\gamma$ in $N$ starting at $p\in M$ which do not meet $M$ in distances $d\in(0,\pi)$, meet each other at the antipode of $p$. The tangent space $T_pN$ contains a unique half space of initial directions of geodesics starting at $p\in M$. If the pre-image of $B(p,\pi)\subset N$ with respect to $\exp_p$ contains the intersection of $B(0,\pi)\subset T_pN$ with this half space, then due to Hopf's maximum principle (see e.g.\ \cite{Esc}), $B(p,\pi)\subset M$ is a geodesic sphere in $\Sp^3$ and the claim is obvious. Otherwise there starts at $p$ a geodesic which touches $M$ for some $t\in(0,\pi)$, and the claim follows in this case. This proves the claim.

In the sequel we consider points $p,\,q \in M$ connected by a geodesic in $N$ with $\dist_N(p,q)<\pi$. Along any geodesic in $\widehat{N}$, which starts at some point $p\in M$ the distance to $M$ is bounded by the geodesic distance to $p$. Therefore such geodesics exist in $\widehat{N}$ at least up to distances not larger than $\frac{\lbcl}{2}$ from the initial point $p\in M$. This shows that in $\widehat{N}$ the absolute value of the second coordinate of $\hat{\gamma}_{\imm}$ is the distance to $M$. Let $(p,q)$ be any pair of points in $M$, which are connected in $\widehat{N}$ by a geodesic. Let $\chi_p,\,\chi_q\in [0,\frac{\pi}{2}]$ denote the angles in $T_pN$ and $T_qN$ between the geodesic $\gamma$ connecting $p$ and $q$ and the normal to $M$, respectively.
 For $\epsilon>0$ to be specified in Lemma~\ref{thm:second derivative} we further distinguish between the following two sub cases of case~{\bf(B)}:
\begin{align*}
\mbox{\bf{(B1)} }
(\sin^2(\chi_p)+\sin^2(\chi_q))^{\frac{1}{2}}&\le\epsilon&&\mbox{and}&
\mbox{\bf{(B2)} }
(\sin^2(\chi_p)+\sin^2(\chi_q))^{\frac{1}{2}}&>\epsilon.
\end{align*}
%
%
%
%

In case {\bf{(B1)}} we shall consider smooth families of geodesics $\gamma$ connecting two smooth paths $s\mapsto p(s)$ and $s\mapsto q(s)$ in $M$ parameterised by a real parameter $s$. For fixed $s$ the geodesic is parameterised by the real parameter $t$. The derivatives with respect to $s$ are denoted by prime and the derivatives with respect to $t$ by dot. For example $p'\in T_{p(s)}M$ and $q'\in T_{q(s)}M$ denotes the tangent vectors along the paths $s\mapsto p(s)$ and $s\mapsto q(s)$. The geodesic $\gamma$ extends in $\Sp^3$ to a closed geodesic. For any $(p',q')\in T_pM\times T_qM$ there exists a Killing field $\vartheta$ on $\Sp^3$, which moves the closed geodesic $\gamma$ in such a way, that the intersection points at $p$ and $q$ moves along $p'$ and $q'$, respectively. Conversely, all Killing fields $\vartheta$ generate a one-dimensional group of isometries of $\Sp^3$. Let $s\mapsto\gamma_\vartheta(s,\cdot)$ denote the corresponding family of geodesics and $s\mapsto(p(s),q(s))$ the corresponding intersection points with $M$. To proceed, we need the following Lemma, whose proof we defer to the appendix.
\begin{lemma}\label{thm:second derivative}
There exist $\epsilon,\delta>0$ and $0<s_0<\min\{\frac{\lbcl}{2},\frac{2}{3}\}$ depending only on $\lbcl$ and $\ubdh$ with the following property: Let $p=p(0),\,q=q(0) \in M$ be connected by a geodesic in $N$ obeying~{\bf{(B1)}}. Then there exists a non-trivial Killing field $\vartheta$, such that $d:s\mapsto d(s)=\dist_{\Sp^3}(p(s),q(s))$ obeys
\begin{align}\label{eq:second derivatives}
d'(s)&\leq 0,& d''(s)&\leq-\delta\cos\left(\tfrac{d(s)}{2}\right),&
|p'(s)|+|q'(s)|&\leq 3\cos\left(\tfrac{d(s)}{2}\right)&
\mbox{ for all }&s\in[0,s_0],
\end{align}
with $s_0 >0$ small enough such that $|d(s_0)-d(0)| \leq \frac{c}{2}$.
\end{lemma}
\emph{Continuation of the proof of Proposition~\ref{thm:ec diameter bound}.} For pairs $(p,q)$ connected in $N$ by a geodesic obeying {\bf{(B1)}} there exists by Lemma~\ref{thm:second derivative} a Killing field $\vartheta$ and two paths $s\mapsto p(s)$ and $s\mapsto q(s)$ along which the length $d$ is reduced for $0\leq s\leq s_0$. We consider the function $f(s)= \cos (\frac{d(s)}{2})$ and its derivative $f'(s) = -\frac{d'(s)}{2}\sin (\frac{d(s)}{2})$ with the inequality
$$-d'(s)\leq|p'(s)|+|q'(s)|\leq 3\cos(\tfrac{d(s)}{2})\leq 3\cos(\tfrac{d(s_0)}{2})$$
to derive $\cos(\frac{d(0)}{2})\geq(1-\frac{3}{2}s_0)\cos(\frac{d(s_0)}{2})$. Twice integration of \eqref{eq:second derivatives} implies the following inequality together with the separate inequalities for the numerator and the denominator:
\begin{align}\label{eq:chord arc 2}
\frac{d(0)-d(s_0)}{\dist_M(p(s_0),p(0))+\dist_M(q(s_0),q(0))}&\geq
\frac{\delta\frac{s_0^2}{2}\cos\left(\frac{d(0)}{2}\right)}
     {3s_0\cos\left(\frac{d(s_0)}{2}\right)}
\geq\delta\frac{s_0}{6}\left(1-\tfrac{3}{2}s_0\right).
\end{align}
The constant $s_0$ is chosen so small such that the geodesic connecting $p(s)$ and $q(s)$ stays inside $\widehat{N}$. We divide the geodesic from $p(s_0)$ to $q(s_0)$ in $\widehat{N}$ into segments outside of $N$ and inside of $N$. In this way an application of Lemma~\ref{thm:second derivative} transforms the geodesic from $p$ to $q$ obeying {\bf{(B1)}} into two paths $s\mapsto p(s)$ from $p(0)=p$ to $p(s_0)$ and $s\mapsto q(s)$ from $q(0)=q$ to $q(s_0)$ in $M$ and several geodesic segments in $\widehat{N}$ from $p(s_0)$ to $q(s_0)$ belonging either to case~{\bf{(A)}}, case~{\bf{(B1)}} or case~{\bf{(B2)}}. Due to~\eqref{eq:chord arc 2} the sum of the lengths of the paths in $M$ times $\delta\frac{s_0}{6}(1-\tfrac{3}{2}s_0)$ plus the lengths of the geodesic segments is smaller than the length of the original geodesic.

In case {\bf{(B2)}} we apply the gradient flow of $\dist_{\widehat{N}}$. As long as $p\ne q$ are connected in $\widehat{N}$ by a geodesic, the function $\dist_{\widehat{N}}$ is smooth. The Riemannian metric on $M\times M$ identifies the negative of the gradient of $\dist_{\widehat{N}}(p,q)$ with a unique vector field $(p',q')$. We follow this flow as long as $(\sin^2(\chi_p)+\sin^2(\chi_q))^{\frac{1}{2}}>\epsilon$ holds and the geodesic connecting $p$ and $q$ stays inside $\widehat{N}$. Along the corresponding integral curves $s\mapsto(p(s),q(s))$ we have $|p'|=\sin(\chi_p)$ and $|q'|=\sin(\chi_q)$. Using the Cauchy-Schwarz inequality, $\dist_{\widehat{N}}(p(s),q(s))$ decreases with derivative
\begin{align}\label{eq:third bound}
-\frac{d}{ds}\dist_{\widehat{N}}(p(s),q(s))=
\sin(\chi_p)|p'(s)|+\sin(\chi_q)|q'(s)|
&\ge\frac{\epsilon}{\sqrt{2}}(|p'(s)|+|q'(s)|).
\end{align}
For such $(p,q)$ the vector field $(p',q')$ is smooth with length bounded from above and below. For finite $s=s_0$ either the geodesic shrinks to point and $p$ becomes equal to $q$, or we reach a pair $(p,q)$ connected by a geodesic in $\widehat{N}$ such that either the geodesic touches the boundary of $\widehat{N}$ or $(\sin^2(\chi_p)+\sin^2(\chi_q))^{\frac{1}{2}}\le\epsilon$ holds. We integrate \eqref{eq:third bound} to
\begin{align}\label{eq:chord arc 3}
\dist_M(p(s_0),p(0))+\dist_M(q(s_0),q(0))&\leq\frac{\sqrt{2}}{\epsilon}
\left(\dist_{\widehat{N}}(p(0),q(0))-\dist_{\widehat{N}}(p(s_0),q(s_0))\right).
\end{align}
Again we decompose the final geodesic in geodesic segments belonging either to case~{\bf{(A)}}, case~{\bf{(B1)}} or case~{\bf{(B2)}}. To sum up the application of the gradient flow transforms the geodesic from $p$ to $q$ obeying {\bf{(B2)}} into two paths $s\mapsto p(s)$ from $p(0)=p$ to $p(s_0)$ and $s\mapsto q(s)$ from $q(0)=q$ to $q(s_0)$ in $M$ and several geodesic segments in $\widehat{N}$ from $p(s_0)$ to $q(s_0)$ belonging either to case~{\bf{(A)}}, case~{\bf{(B1)}} or case~{\bf{(B2)}}. Due to~\eqref{eq:chord arc 3} the sum of the lengths of paths in $M$ times  $\frac{\epsilon}{\sqrt{2}}$ plus the lengths of the geodesic segments is smaller than the length of the original geodesic.

We iterate the applications of \eqref{eq:chord arc 1} to pairs $(p,q)$ of case~{\bf{(A)}}, the applications of Lemma~\ref{thm:second derivative} with \eqref{eq:chord arc 2} to pairs $(p,q)$ of case~{\bf{(B1)}} and the applications of the gradient flow with \eqref{eq:chord arc 3} to pairs of case~{\bf{(B2)}}. Any application of Lemma~\ref{thm:second derivative} reduces the length $\dist_{\widehat{N}}(p,q)$ by a number $\ge \tfrac{1}{2}\delta s_0^2\cos(\frac{d}{2})$, where $0<d<\pi$ denotes the length of the original geodesic. Any application of the gradient flow which ends at a pair obeying {\bf{(B1)}} results in an application of Lemma~\ref{thm:second derivative}. All other applications of the gradient flow reduces $\dist_{\widehat{N}}(p,q)$ by a number $\ge\frac{\epsilon\lbcl}{2\sqrt{2}}$. Therefore finitely many applications of Lemma~\ref{thm:second derivative} and the gradient flow transform the geodesic connecting $p$ and $q$ into finitely many paths in $M$, which connect $p$ with $q$ and whose total length is bounded by $\cab\dist_N(p,q)$ with $\cab=\max\{2,\,\frac{12}{\delta(2s_0-3s_0^2)},\,\frac{\sqrt{2}}{\epsilon}\}$.
\end{proof}
We remark that this proof also shows the equivalence of
\eqref{eq:chord arc 0} and \eqref{eq:chord arc} with $\cab\ge 2$ for
mean-convex Alexandrov embeddings $\imm:M\to\Sp^3$ with lower bound
$\lbcl$ on the cut locus function.
%
%
\section{Collar perturbation}
In this section we localize the concept of a mean-convex Alexandrov embedding. We consider in the following open subsets $V$ of $M$ and open bounded 3-dimensional manifolds $W$ with boundary $V$. More precisely, we denote by $\partial W=V$ the boundary as defined within the concept of manifolds with boundary. It is in general a subset of the topological boundary.
\begin{definition}
We call the restriction $\imm|_V$ of $\imm:M\to\Sp^3$ to an open subset $V \subset M$ a local mean-convex Alexandrov embedding if $\imm|_V$ extends as an immersion to an open 3-manifold $W$ with $V = \partial W$ such that the following hold:

{\rm{(i)}} The mean curvature of $V$ in $\Sp^3$ with respect to the inward normal is non-negative everywhere.

{\rm{(ii)}} All inward $V$-geodesics exist in $W$ until they reach
  the cut locus \eqref{eq:cutlocus} (for $t\le\frac{\pi}{2}$).

{\rm{(iii)}} $ W=\{\gamma_{\imm}(p,t) \mid p\in V\mbox{ and }0 < t\leq c_{\imm}(p)\}$.
\end{definition}
If $\imm:M\to\Sp^3$ is a mean-convex Alexandrov embedding which satisfies the chord-arc bound \eqref{eq:chord arc 1}, then all open subsets $V\subset M$ such that for all $p\in V$ the classes $[p]_{\imm} \subset V$ are examples of local mean-convex Alexandrov embeddings. To see that we only have to show that
$$
    W=\{\gamma_{\imm}(p,t)\in N\mid p\in V\mbox{ and }0 < t\leq c_{\imm}(p)\}
$$
is open in $N$. Since $V$ is open, $W$ is open at all points away from the cut locus. Consider a cut point $c_{\imm}(p)$ and a sequence of  $(q_n) \in N$ converging to $c_{\imm}(p)$. We prove that $(q_n) \in W$ for n large enough. Let $q_n = \gamma_{\imm}(p_n,t) \in N$ with $p_n \in M$. By the chord-arc bound the set $[p]_{\imm}$ is bounded in $M$, and there is a subsequence of $p_n$ converging to an element of $[p]_{\imm}\subset V$. This proves that $p_n \in V$ for $n$ large enough and $q_n \in W$. Hence $W$ is an open neighborhood around cut points. Since $W$ is an open subset of $N$ with $V=\partial W$, $\imm_{| V}$ extends naturally as an immersion to $W$ by restriction of the extension of $f$ to $W$.

We shall prove that `mean-convex Alexandrov embeddedness' is an open condition, which will allow us to study deformation families of mean-convex Alexandrov embeddings. The main tool is a general perturbation technique of Alexandrov embeddings, which we call collar perturbation. We consider local perturbations $\tilde{\imm}$ of a given smooth immersion $\imm:M\to\Sp^3$, which are `small' with respect to the $C^1$-topology on the space of immersions from $M$ into $\Sp^3$.
\begin{lemma}\label{thm:collar deformation 1}
For given $\lbcl>0$ and $\ubdh>0$ there exist $\epsilon>0$ and $R>0$ with the following property: Let $\imm:M\to\Sp^3$ be a mean-convex Alexandrov embedding with cut locus function $c_{\imm}$ \eqref{eq:cutlocus} bounded from below by $\lbcl$ and second fundamental form obeying \eqref{eq:upper bound derivative of h}, and let $\tilde{\imm}:M \to\Sp^3$ be an immersion with non-negative mean curvature and principal curvatures bounded by $\kappa\ind{max}=\cot(\lbcl)$. Furthermore, let $p\in M$ be some point and let both immersions $\imm$ and $\tilde{\imm}$ obbey on the ball $B(p,2R)\subset M$ with respect to the metric induced by $\tilde{\imm}$ the following estimate:
\begin{equation}\label{eq:immersion bound}
    \|\,\imm -\tilde{\imm}\,\|_{\mathrm{C}^1(B(p,2R))} < \epsilon
\end{equation}
Then the restriction $\tilde{\imm}|_V$ of $\tilde{\imm}$ to an open neighbourhood $V\subset M$ of $p$ extends to a local mean-convex Alexandrov embedding with $V=\partial W$.
\end{lemma}
\begin{proof}
In the following discussion we use the construction in section~\ref{sec:chord arc} of a larger complete 3-manifold $\widehat{N}\supset N$ with boundary, such that the generalised cylinder coordinates \eqref{eq:cylinder} extend to an embedding $\hat{\gamma}_{\imm}:M\times[-\frac{\lbcl}{2},\lbcl)\hookrightarrow\widehat{N}$ and the immersion $\imm:N\to\Sp^3$ extends to an immersion $\hat{\imm}:\widehat{N}\to\Sp^3$. The restriction $\tilde{\imm}|_{B(p,R)}$ to $B(p,R)$ of an immersion $\tilde{\imm}$ satisfying~\eqref{eq:immersion bound} defines a submanifold $O\subset\widehat{N}$. We identify $\tilde{\imm}|_{B(p,R)}$ with the immersion $\hat{\imm}|_O$. The immersions $\imm$ and $\tilde{\imm}$ induce on $M$ two Riemannian metrics whose distance functions are denoted by $\dist_M$ and $\dist_{\tilde{M}}$, respectively. Again $\dist_{\widehat{N}}$ denotes the distance function of the complete Riemannian manifold $\widehat{N}$. We claim that the following estimate holds for some $\Cab>0$:
\begin{align}\label{eq:distance bound}
\dist_{\tilde{M}}(q,q')&\le\Cab\dist_{\widehat{N}}(q,q')&\mbox{for all }q,q'&\in O.
\end{align}
The bound~\eqref{eq:immersion bound} implies that for $q,q'\in O$ the corresponding points of the immersion $\imm$ are contained in $B(q,\epsilon)\cap M$ and $B(q',\epsilon)\cap M$. Moreover, the two Riemannian metrics induced by $\imm$ and $\tilde{\imm}$ on $B(p,2R)$ are uniformly bounded in terms of each other. The shortes path with respect to the latter Riemannian metric which connects $p,q\in B(p,R)$ is contained in $B(p,2R)$. We conclude $\dist_{\tilde{M}}(p,q)\le C_2\dist_M(p,q)$ with $C_2>0$ depending on $\epsilon$. Proposition~\ref{thm:ec diameter bound} implies
$$\dist_{\tilde{M}}(q,q')\le C_2\cab(\dist_{\widehat{N}}(q,q')+2\epsilon).$$
This implies \eqref{eq:distance bound} for $\dist_{\widehat{N}}(q,q')\ge2\epsilon$ with $\Cab=2C_2\cab$.

Let $D(q,r)$ be a totally geodesic disc in $\widehat{N}$ of radius $r>0$ tangent to $O$ at the center $q$. The curvature of $\tilde{\imm}$ is uniformly bounded. For sufficently small $r$ the corresponding surface is locally in $\widehat{N}$ a bounded exponential normal graph over $D(q,r)$. The gradient of the corresponding function is at $q'\in D(q,r)$ bounded by $C_3\dist_{D(q,r)}(q,q')$ with $C_3>0$. For sufficiently small $\epsilon>0$ all $q'\in O$ with $\dist_{\widehat{N}}(q,q')<2\epsilon$ belong to this graph and obey due to the gradient bound
$$
\dist_{\tilde{M}}(q,q') \leq 2\dist_{\widehat{N}}(q,q') \quad \mbox{ for all }q,q' \in O
\quad \mbox{ with }\dist_{\widehat{N}}(q,q')<2\epsilon.
$$
Therefore~\eqref{eq:distance bound} holds with $\Cab>0$ and $\epsilon>0$ depending only on $\lbcl$ and $\ubdh$.

Since the principal curvatures of $\tilde{\imm}$ are bounded by $\kappa\ind{max}$ the inward $O$-geodesics are definded in $\widehat{N}$ up to distances $\frac{\lbcl}{2}$ and define an immersion $\gamma_{\tilde{\imm}}:O\times[0,\frac{\lbcl}{2})\hookrightarrow\widehat{N}$. We define
$$
    U=\{\gamma_{\tilde{\imm}}(q,t)\mid
    q\in O\mbox{ and }0\leq t<\tfrac{\lbcl}{2}\}\cup
    \{\gamma_{\imm}(q,t)\mid q\in M\mbox{ and }
    \epsilon<t\leq c_{\imm}(q)\}.
$$
We choose $\frac{\epsilon}{\lbcl}$ small enough such that $U$ contains the subset $\{\gamma_{\tilde{\imm}}(q,\tfrac{\lbcl}{2})\mid q\in O\}$ of the closure of the first set in this union. Let $\dist_U$ denote the distance function of this non-complete Riemannian manifold $U$ with boundary $O$. Since $U$ is a submanifold of $\widehat{N}$ we have $\dist_{\widehat{N}}(p,q)\le\dist_U(p,q)$ for $p,q\in U$, and \eqref{eq:distance bound} implies the chord-arc bound
\begin{align}\label{eq:chord arc 4}
\dist_{\tilde{M}}(q,q')&\le\Cab\dist_{\widehat{N}}(q,q')\le\Cab\dist_U(q,q')&\mbox{for all }q,q'&\in O.
\end{align}
All cut locus functions of mean-convex Alexandrov embeddings are uniformly bounded from above by $\frac{\pi}{2}$, since otherwise a sphere with negative principal curvatures would touch $M$ inside of $N$ contradicting Hopf's maximum principle. If the inward $O$ geodesics of $p,q\in O$ intersect at distances not larger than $\frac{\pi}{2}$, then $\dist_{\tilde{M}}(p,q)\le\Cab\pi$. Let$R=3\Cab(\pi+\epsilon)$. For $r\le R$ we define $O(r)\subset O$ as the subset such that $\hat{\imm}|_{O(r)}$ is identified with the restriction $\tilde{\imm}|_{B(p,r)}$ to the ball $B(p,r)\subset M$ with respect to $\dist_{\tilde{M}}$. For $q\in O(\frac{R}{3})$ we have
$$
\left\{ q'\in O \mid
\exists t \in [0,\,\tfrac{\pi}{2}]
\mbox{ with }\dist_{U}(\gamma_{\tilde{\imm}}(q,t),q') \leq t \right\}
\subset \left\{ q' \in O \mid \dist_{U}(q,q') \leq \pi
\right\} \subset O(\tfrac{2}{3}R)\,.
$$
Therefore, for all $q\in O(\frac{R}{3})$ the cut locus function $c_{\tilde{\imm}}$ is well defined. For all such $q\in O(\frac{R}{3})$, let $[q]_{\tilde{\imm}}$ denote the set
$$
[q]_{\tilde{\imm}}=\left\{q'\in O\mid
\dist_U\left(\gamma_{\tilde{\imm}}(q,c_{\tilde{\imm}}(q)),q' \right)=
c_{\tilde{\imm}}(q)\right\}.
$$
For any closed subset $A\subset O$ the set $\{q\in O(\frac{R}{3})\mid [q]_{\tilde{\imm}}\cap A\not=\emptyset\}$ is a closed subset of $O(\frac{R}{3})$. Hence $V=\{q\in O(\frac{R}{3})\mid [q]_{\tilde{\imm}}\subset O(\frac{R}{3})\}$ is an open subset of $O$ which contains $p$ by the choice of $R$. Furthermore $W=\{\gamma_{\tilde{\imm}}(q,t)\mid q\in V \mbox{ and }0\leq t\leq c_{\tilde{\imm}}(q)\}$ is a subset of $\widehat{N}$ with boundary $V$. By construction $\hat{\imm}_{|V}$ is a local mean-convex Alexandrov embedding with $[q]_{\tilde{\imm}} \subset V$.
\end{proof}
\section{From local to global mean-convex Alexandrov embeddings}
In this section we consider an immersion $\tilde{\imm} : M \to \Sp^3$ which is covered by a family of local mean-convex Alexandrov embedded neighborhoods $\tilde{\imm}_p: V_p \to \Sp^3$ of $p\in M$. We explain how to extend $\tilde{\imm}: M \to \Sp^3$ to the 3-manifold $\tilde{N}=\cup_{p\in M}W_p$, with $\partial \tilde{N}=M$. The set of local immersions $\tilde{\imm}_p$ are constructed by local collar perturbations.
\begin{proposition}\label{thm:collar deformation 2} For given $\lbcl>0$ and $\ubdh>0$ there exist $\epsilon>0$ and $R>0$ with the following property: Let $\tilde{\imm}:M\to\Sp^3$ be an immersion with non-negative mean curvature and principal curvatures bounded by $\kappa\ind{max}=\cot(\lbcl)$ such that for all $p\in M$ there exists a mean-convex Alexandrov embedding $\imm_p:M\to\Sp^3$ with cut locus function \eqref{eq:cutlocus} bounded from below by $\lbcl$ and second fundamental form obeying \eqref{eq:upper bound derivative of h}. If $\imm_p$ and $\tilde{\imm}$ obey for all $p\in M$ \eqref{eq:immersion bound} on the ball $B(p,2R)$ with respect to the metric induced by $\tilde{\imm}$, then $\tilde{\imm}$ extends to a mean-convex Alexandrov embedding.
\end{proposition}
\begin{proof} Let $\Cab$ denote the same constant as in Lemma~\ref{thm:collar deformation 1}. We apply Lemma~\ref{thm:collar deformation 1} to all $p\in M$ with the same $R=3\Cab(\pi+\epsilon)$ and decorate the corresponding objects with an index $p$. The balls $B(p,R)$ are embedded as Riemannian submanifolds $O_p\hookrightarrow\widehat{N}_p$. We obtain a covering of $M$ by open subsets $V_p=\{q\in O_p(\frac{R}{3})\mid [q]_{\tilde{\imm}}\subset O_p(\frac{R}{3})\}$ of the balls $O_p(\frac{R}{3})=B(p,\frac{R}{3})$ with respect to the metric induced by $\tilde{\imm}$. Finally, the restrictions $\tilde{\imm}|_{V_p}$ of $\tilde{\imm}$ to the members $V_p$ of this covering extend to local mean-convex Alexandrov embeddings $\tilde{\imm}_p:W_p\to\Sp^3$ with open subsets $W_p\subset U_p$.

We shall glue the manifolds $(W_p)_{p\in M}$ to obtain a 3-manifold $\tilde{N}$ with boundary $M$. In order to glue $W_p$ with $W_q$ for $\dist_M(p,q)>\epsilon$ we choose finitely many intermediate points $p=p_0,\ldots,p_L=q$ with distances $\dist_M(p_{l-1},p_l)<\epsilon$ for $l=1,\ldots L$. Therefore it suffices to consider $p,q\in M$ with $\dist_M(p,q)<\epsilon$. The intersection $O_{pq}(\frac{R}{3})=O_p(\frac{R}{3})\cap O_q(\frac{R}{3})$ is a connected subset of the Riemannian manifold $M$ with the metric induced by $\tilde{\imm}$. The immersions $\tilde{\imm}_p$ and $\tilde{\imm}_q$ define on $U_p \subset \widehat{N}_p$ and $U_q \subset \hat{N}_q$ chord-arc distance $\dist_{U_p}(p',q')$ and $\dist_{U_q}(p',q')$. Now define $R(t)=\Cab(3(\pi + \epsilon)-t)$, and set $\bar{O}_{pq}(R(t))=\bar{O}_p(R(t))\cap\bar{O}_q(R(t))$ as the intersection of the closures of the related shrunk open sets. Note that $\bar{O}_{pq}(R(t))\subset O_p\cap O_q$ for $t>0$. Denote
\begin{align*}
    A_p(t) &=\{  (p',q') \in \bar{O}_{pq}(R(t))\times \bar{O}_{pq}(R(t))\mid \dist_{U_p}(p',q') = \dist_{\Sp^3}(\tilde{\imm}( p') ,\tilde{\imm}(q') )=t \}\,,\\
    A_q(t) &=\{  (p',q') \in \bar{O}_{pq}(R(t))\times  \bar{O}_{pq}(R(t)) \mid \dist_{U_q}(p',q') = \dist_{\Sp^3}(\tilde{\imm}( p') ,\tilde{\imm}(q') )=t \}\,.
\end{align*}
If $(p',q') \in A_p(t)$ then $p'$ and $q'$ are connected by a geodesic segment lying in $U_p$ which is mapped by $\tilde\imm_p$ to the unique shortest geodesic in $\Sp^3$ from $\tilde\imm(p')$ to $\tilde\imm(q')$.
By Lemma \ref{thm:collar deformation 1} the points in $O_p$ obey in $U_p$ a chord-arc bound~\eqref{eq:chord arc 4}. If the geodesic segment in $U_p$ from $p'$ to $q'$ meets a point $p'' \in M$, then $\dist_{U_p}(p'',q')=\dist_{U_p}(p',q')-\dist_{U_p}(p',p'')$, and \eqref{eq:chord arc 4} implies
\begin{align*}
\dist_{\tilde{M}}(p,p'')&\le\dist_{\tilde{M}}(p,q')+\dist_{\tilde{M}}(q',p'')\\
&\le R(\dist_{U_p}(p',q'))+\Cab(\dist_{U_p}(p',q')-\dist_{U_p}(p',p''))=
R(\dist_{U_p}(p',p''))\\
\dist_{\tilde{M}}(q,p'')&\le\dist_{\tilde{M}}(q,p')+\dist_{\tilde{M}}(p',p'')\\
&\le R(\dist_{U_p}(p',q'))+\Cab(\dist_{U_p}(p',q')-\dist_{U_p}(p',p''))=
R(\dist_{U_p}(p',p'')).
\end{align*}
With $\dist_{U_p}(p',p'')=\dist_{\Sp^3}(\tilde{\imm}(p'),\tilde{\imm}(p''))$ and with interchanged $p'$ and $q'$ these arguments show
\begin{align}\label{eq:implication}
(p',p'')&\in A_p(\dist_{U_p}(p',p''))&&\mbox{and}&
(p'',q')\in A_p(\dist_{U_p}(p'',q')).
\end{align}

\noindent{\it Claim.} $A_p(t)=A_q (t)$ for all $t \in [0,\pi]$.

We shall prove this claim later, and first show that it implies that $c_{\tilde{\imm}_p}$ and $c_{\tilde{\imm}_q}$ coincide on $V_p\cap V_q$. Let us assume on the contrary $c_{\tilde{\imm}_p}(p')<c_{\tilde{\imm}_q}(p')$ for $p'\in V_p\cap V_q$. The domain $U_p$ contains a ball of radius $c_{\tilde{\imm}_p}(p')$ centered at $\gamma_{\tilde{\imm}_p} (p', c_{\tilde{\imm}_p}(p'))$ (analogously with $U_q$ and $c_{\tilde{\imm}_q}(p')$).

Then there exists $p'\ne q'\in [p']_{\tilde{\imm}_p}\subset O_p(\frac{R}{3})\cap O_q(\frac{2R}{3})$. Then $p'$ and $q'$ are connected by a segment of a geodesic in $U_p$, which meets the boundary only at the end points $p'$ and $q'$ by strict convexity of a geodesic ball centered at the cut point. Therefore there does not exists a point $p''\in O_p$ with $\dist_{U_p}(p',q')=\dist_{U_p}(p',p'')+\dist_{U_p}(p'',q')$. Due to the claim there does not exist $p''\in O_q(\frac{2R}{3})$ with $\dist_{U_q}(p',q')=\dist_{U_q}(p',p'')+\dist_{U_q}(p'',q')$. Therefore the shortest path in $U_q$ connecting $p'$ and $q'$ is also a segment of a geodesic, which meets the boundary only at the end point $p'$ and $q'$. Furthermore, both immersions $\hat{\imm}_p$ and $\hat{\imm}_q$ map these segments onto the same geodesic in $\Sp^3$ connecting $\tilde{\imm}(p')$ and $\tilde{\imm}(q')$. Since $p'$ and $q'$ both belong to $[p']_{\tilde{\imm}_p}$ there exists a unique geodesic 2-sphere in $\Sp^3$, which intersects the image of $\tilde{\imm}$ orthogonally at $\tilde{\imm}(p')$ and at $\tilde{\imm}(q')$. The pre-image with respect to $\hat{\imm}_q$ of this 2-sphere in $U_q$ intersects $O_q$ orthogonally at $p'$ and $q'$, and contains a segment of a geodesic connecting $p'$ and $q'$. This segment of a geodesic is in $U_q$, because it is contained in a geodesic ball of larger radius and centered at the cut point of $p'$ in $U_q$. This implies that the inward $O_q$-geodesics at $p'$ and $q'$ also is contained in the larger geodesic ball of $U_q$ and they meet each other at distance $c_{\tilde{\imm}_p}(p')=c_{\tilde{\imm}_q}(q')$ in contradiction to $c_{\tilde{\imm}_p}(p')<c_{\tilde{\imm}_q}(p')$. Interchanging $p$ and $q$ we get the other inequality, and thus both cut locus functions coincide:
\begin{align}\label{eq:cut locus equal}
c_{\tilde{\imm}_p}(p')&=c_{\tilde{\imm}_q}(p')&\mbox{for all }p'\in V_p\cap V_q.
\end{align}

We claim that the open set $V_p \cap V_q$ extends in two different but isometric extensions $Z_p\subset W_p$ and $Z_q\subset W_q$ as local mean-convex Alexandrov embeddings with
\begin{align*}
    Z_p &=\{\gamma_{\tilde{\imm}_p}(p',t)\in W_p\mid p' \in V_p\cap V_q
    \mbox{ and }0 < t\leq c_{\tilde{\imm}_p}(p')\}\,,\\
    Z_q &=\{\gamma_{\tilde{\imm}_q}(q',t)\in W_q\mid q' \in V_p\cap V_q
    \mbox{ and }0 < t\leq c_{\tilde{\imm}_q}(q')\}\,.
\end{align*}
We identify $Z_p$ with $Z_q$ by a bijection $\psi: Z_p \to Z_q$ with trivial restriction $\psi|_{V_p\cap V_q}$ to $V_p\cap V_q$ such that $\tilde{\imm}_p|_{Z_p}=\tilde{\imm}_q|_{Z_q}\circ\psi$.  By definition of the generalised cylinder coordinates~\eqref{eq:cylinder} this is equivalent to
\begin{align}\label{eq:bijection}
\psi \circ \gamma_{\tilde\imm_p}(p',t)&=\gamma_{\tilde{\imm}_q}(p',t)&\mbox{ for all }p'&\in V_p\cap V_q.
\end{align}
Due to~\eqref{eq:cut locus equal} condition~\eqref{eq:bijection} uniquely defines $\psi$. Indeed ~\eqref{eq:cut locus equal} implies that for $p'\in V_p\cap V_q$ the equivalence classes $[p']_{\tilde{\imm}_p}$ and $[p']_{\tilde{\imm}_q}$ \eqref{eq:ec cutlocus} of the local mean-convex Alexandrov embeddings $\tilde{\imm}_p:W_p\to\Sp^3$ and $\tilde{\imm}_q:W_q\to\Sp^3$ coincide. Hence $\psi$ is well defined.

Since $\tilde{\imm}_p$ and $\tilde{\imm}_q$ induce the Riemannian metrics of $Z_p$ and $Z_q$ any $\psi$ obeying~\eqref{eq:bijection} is distance preserving. By a theorem of Myers-Steenrod \cite{MyeS} (see also \cite{petersen}, Theorem 18, p.\ 147) such distance preserving bijections bewteen Riemannian manifolds are diffeomorphisms.

The union of mean-convex Alexandrov embeddings extend to a manifold $\tilde{N}$. It remains to show that $\tilde{N}$ is complete with respect to the Riemannian metric induced by $\tilde{\imm}$. By construction, every point of the Riemannian manifold $\tilde{N}$ with the metric induced by $\tilde{\imm}$ is the center of an $\epsilon$-ball contained in one of the complete manifolds $\widehat{N}_p$. Therefore $\tilde{N}$ is complete.

{\it Proof of the Claim.} We consider the set $B$ of all $t_0\in[0,\pi]$ such that $A_p(t) = A_q(t)$ holds for all $t \in [0,\,t_0]$. We prove that $B$ is both closed and open in $[0,\,\pi]$, and thus $B=[0,\pi]$. Due to the lower bound $\lbcl$ of the cut locus function, $B$ contains the set $[0,\,\lbcl - \epsilon)$.

Set $t_0= \sup B$, and suppose $(p',\,q') \in A_p(t_0)$. We need to consider two cases. In the first case we assume that the unique geodesic in $U_p$ from $p'$ to $q'$ goes through a point $p'' \in O_p\setminus \{p',\,q'\}$. As in~\eqref{eq:implication} this implies for $t=\dist_{U_p}(p',p'')<t_0$ and $t'=\dist_{U_p}(p'',q')<t_0$
\begin{align*}
(p',p'') \in A_p(t) &= A_q(t)&&\mbox{and}&
(p'',q') \in A_p(t') &= A_q(t').
 \end{align*}
Both geodesics in $U_q$ which connect $p'$ to $p''$, and $p''$ to $q'$, are mapped by $\tilde{\imm}_q$ onto the unique geodesic in $\Sp^3$ connecting $\tilde{\imm}(p')$ to  $\tilde{\imm}(q')$. This implies that $p'$ and $q'$ are connected in $U_q$ by a smooth geodesic, and hence $(p',q') \in A_q(t_0)$. This proves $A_p(t_0) \subset A_q(t_0)$, and analogously also the other inclusion, and therefore $A_p(t_0) = A_q(t_0)$.

For the second case we may assume that there is no point $p'' \in O_p\setminus \{p',q'\}$ on the geodesic in $U_p$ from $p'$ to $q'$.
The mean-convexity of the surface implies that $(p',q')$ is not a local minimum of the function $\dist_{U_p}$ on
$O_p\times O_p$ (see Lemma~\ref{thm:second derivative}).
Then there exists a sequence of $(p_n,q_n)$ which converges to $(p',q')$ such that $\dist_{U_p}(p_n,q_n)<t_0$.
The continuity of $\dist_{U_p}$ and $\dist_{U_q}$ imply that
\[
\dist_{U_q}(p',q') = \lim_{n \to \infty} \dist_{U_q} (p_n,q_n)
= \lim_{n \to \infty} \dist_{U_p} (p_n,q_n) = \dist_{U_p}(p',q')\,.
\]
Hence $(p',q') \in A_q(t_0)$, and thus $A_p(t_0) \subset A_q(t_0)$. The other inclusion is obtained by switching the roles of $p$ and $q$. Therefore  $A_p(t_0) = A_q(t_0)$ in both cases, which shows $t_0 \in B$, and proves that $B$ is closed.

We now show that $B$ is open. If the maximum $t_0$ is smaller than
$\pi$, then there exists a sequence $t_n \in (t_0,\,\pi]$ which converges to $t_0$ such that $A_p(t_n) \neq A_q(t_n)$. By passing to a subsequence we may assume without loss of generality that there exists $(p_n,q_n) \in A_p(t_n)$ but $(p_n,q_n) \notin A_q(t_n)$. A subsequence of $(p_n,q_n)$ converges to $(p',q') \in A_p(t_0) = A_q(t_0)$. Therefore $p'$ and $q'$ are connected by smooth geodesic segments in both $U_p$ and $U_q$. Both of these geodesic segments are mapped by $\tilde\imm_p$ and $\tilde\imm_q$ to the unique shortest geodesic in $\Sp^3$ from $\tilde{\imm}(p')$ to $\tilde{\imm}(q')$. Furthermore, both these geodesic segments meet the boundaries $O_p$ and $O_q$ in the same points. The balls of radii $\lbcl - \epsilon$ around each of such boundary points in $U_p$ and $U_q$ are isometric. In the complement of these balls both geodesic segments have positive distances to the boundaries  $O_p$ and $O_q$. Hence two tubular neighbourhoods of these geodesic segments in  $U_p$ and $U_q$ are also isometric.

For large $n$ the geodesic segments in $U_p$ connecting $p_n$ with $q_n$ belong to the tubular neighbourhood in $U_p$. They are isometric to geodesic segments in $U_q$ connecting $p_n$ with $q_n$. This implies $(p_n,q_n) \in A_q(t_n)$ for large $n$. This contradicts the assumption that $t_0 < \pi$. Hence we have that $B=[0,\,\pi]$, and the claim is proven.
\end{proof}
\section{Appendix - Proof of Lemma~\ref{thm:second derivative}}
\begin{proof} We shall construct a Killing field $\vartheta$ with the desired properties, which rotates $\gamma$ around two antipodes of $\gamma$. The corresponding rotated geodesics $\gamma_\vartheta(s,\cdot)$ belong to a unique great 2-sphere $\Sp^2\subset\Sp^3$. The corresponding paths $s\mapsto p(s)$ and $s\mapsto q(s)$ move along the intersection of this 2-sphere $\Sp^2$ with $M$. Hence we can calculate all derivatives on this sphere.

We parameterise this 2-sphere by the real parameter $s$ of the family $s\mapsto\gamma_{\vartheta}(s,\cdot)$ of rotated geodesics, and the real arc length parameter $t$ of these geodesics. We choose the equator as the points corresponding to $t=0$ with distance $\frac{\pi}{2}$ to the rotation axis. Let $t_p$ and $t_q$ denote the values of this parameter $t$ at the points $p(s)$ and $q(s)$. Hence $\dist_N(p,q)=|t_p-t_q|$. The vector fields $\vartheta$ and the geodesic vector field $\dot{\gamma}$ along the geodesics $\gamma_\vartheta(s,\cdot)$ form an orthogonal basis of the tangent spaces of this 2-sphere away from the zeroes of $\vartheta$. The vector fields $\vartheta$ and $\dot{\gamma}$ have at $(s,t)$ the scalar products
\begin{align*}
g(\vartheta,\vartheta)&=\cos^2(t)\,,&
g(\vartheta,\dot{\gamma})&=0\,,&
g(\dot{\gamma},\dot{\gamma})&=1\,.
\end{align*}
Since $\dot{\gamma}$ is a geodesic vector field the derivative $\nabla_{\dot{\gamma}}\dot{\gamma}$ vanishes. Moreover, the geodesic curvature
in $\Sp^2$ of the integral curve of $\vartheta$ starting at $(s,t)$ is equal to $\tan(t)$. Therefore at $(s,t)$ we have
\begin{align*}
    \nabla_{\vartheta}\vartheta&
    =\cos^2(t)\tan(t)\dot{\gamma}
    =\cos(t)\sin(t)\dot{\gamma}\,,&
    \nabla_{\vartheta}\dot{\gamma}&
    =-\tan(t)\vartheta \,,\\
    \nabla_{\dot{\gamma}}\vartheta&
    =-\tan(t)\vartheta \,,&
    \nabla_{\dot{\gamma}}\dot{\gamma}& =0\,.
\end{align*}
We parameterise a neighbourhood of the geodesic from $p$ to $q$ in such a way that the corresponding vector field $\dot{\gamma}$ points inward to $N$ at $p$ and outward of $N$ at $q$, respectively. The derivatives of $s\mapsto p(s)$ and $s\mapsto q(s)$ are equal to
\begin{align*}
    p'&=\vartheta(p)-\dot{\gamma}(p)
    \frac{g(\normal(p),\vartheta(p))}
     {g(\normal(p),\dot{\gamma}(p))}\quad\mbox{ and }&
    q'&=\vartheta(q)-\dot{\gamma}(q)
    \frac{g(\normal(q),\vartheta(q))}
     {g(\normal(q),\dot{\gamma}(q))}.
\end{align*}
The lengths $|p'|$ and $|q'|$ depend on the angles $\sphericalangle(\normal(p),\vartheta(p))$ and $\sphericalangle(\normal(q),\vartheta(q))$. Since $\vartheta$ is orthogonal to $\dot{\gamma}$ and $\sphericalangle(\normal(p),\dot{\gamma}(p))=\chi_p$ and $\sphericalangle(\normal(q),\dot{\gamma}(q))=\chi_q$ these angles obey $\sphericalangle(\normal(p),\vartheta(p))\in[\frac{\pi}{2}-\chi_p,\frac{\pi}{2}+\chi_p]$ and $\sphericalangle(\normal(q),\vartheta(q))\in[\frac{\pi}{2}-\chi_q,\frac{\pi}{2}+\chi_q]$. Then we have
\begin{align}\label{eq:tangent bound}
|p'|&\le\frac{|\cos(t_p)|}{\cos(\chi_p)}&
|q'|&\le\frac{|\cos(t_q)|}{\cos(\chi_q)}.
\vspace{-5mm}\end{align}
$$
    d'=
    \frac{g(\normal(p),\,\vartheta(p))}
         {g(\normal(p),\,\dot{\gamma}(p))}-
    \frac{g(\normal(q),\,\vartheta(q))}
         {g(\normal(q),\,\dot{\gamma}(q))}=
    \frac{g(\normal(p),\,\vartheta(p))}
         {\cos(\chi_p)}-
    \frac{g(\normal(q),\,\vartheta(q))}
         {\cos(\chi_q)}.
$$
Along the paths $p$ and $q$ with $X=p'$ and $X=q'$, respectively, we have at $(s,t)$
\begin{align*}
    \nabla_X\frac{g(\normal,\,\vartheta)}
    {g(\normal,\,\dot{\gamma})}&=
    \frac{g(\nabla_X\normal,\,\vartheta)+
    g(\normal,\,\nabla_X\vartheta)}
    {g(\normal,\,\dot{\gamma})}-
    \frac{g(\normal,\vartheta)(
    g(\nabla_X\normal,\,\dot{\gamma})+
    g(\normal,\,\nabla_X\dot{\gamma}))}
    {(g(\normal,\,\dot{\gamma}))^2}\\
    &=\frac{g(\nabla_X\normal,X)+g(\normal,\nabla_X\vartheta)}
           {g(\normal,\dot{\gamma})}-
      \frac{g(\normal,\vartheta)
                g(\normal,\nabla_X\dot{\gamma})}
           {g(\normal,\dot{\gamma})^2}\\
    &=-\frac{\II(X,\,X)}{g(\normal,\dot{\gamma})}
    +\cos(t)\sin(t)
    +2\tan(t)\left(\frac{g(\normal,\,\vartheta)}
                        {g(\normal,\,\dot{\gamma})}\right)^2.
\end{align*}
Hence the second derivative is equal to
\begin{align*}
d''&=-\frac{\II(p',\,p')}{\cos(\chi_p)}-\frac{\II(q',\,q')}{\cos(\chi_q)}
+\frac{\sin(2t_p)-\sin(2t_q)}{2}\\
   &+2\tan(t_p)\left(\frac{g(\normal(p),\,\vartheta(p))}
         {g(\normal(p),\,\dot{\gamma}(p))}\right)^2
-2\tan(t_q)\left(\frac{g(\normal(q),\,\vartheta(q))}
         {g(\normal(q),\,\dot{\gamma}(q))}\right)^2.
\end{align*}
If along the rotation of the geodesic for $s\in[0,s_0]$ the following inequalities are satisfied
\begin{align}\label{eq:assumption 1}
-\tfrac{\pi}{2}\leq t_p&\leq 0,&
0\leq t_q&\leq\tfrac{\pi}{2},&
\tfrac{\lbcl}{2}&\leq d=t_q-t_p,\quad\mbox{ and }&
(\sin^2(\chi_p)+\sin^2(\chi_q))^{\frac{1}{2}}&\leq\tfrac{1}{2},
\end{align}
then $\min\{\cos(\chi_p),\cos(\chi_q)\}\geq\tfrac{\sqrt{3}}{2}$ implies the third inequality of \eqref{eq:second derivatives}:
\begin{align*}
    |p'|+|q'|&\le\frac{\cos(t_p)}{\cos(\chi_p)}
              +\frac{\cos(t_p)}{\cos(\chi_p)}
    \leq\tfrac{2}{\sqrt{3}}(\cos(t_p)+\cos(t_q))
    =\tfrac{4}{\sqrt{3}}\cos\left(\tfrac{d}{2}\right)
    \cos\left(\tfrac{t_p+t_q}{2}\right)\leq3\vn\cos(\tfrac{d}{2}).
\end{align*}
Furthermore, the last two terms of $d''$ are bounded by
\begin{align*}
\left|\tan(t_p)\left(\frac{g(\normal(p),\,\vartheta(p))}
         {g(\normal(p),\,\dot{\gamma}(p))}\right)^2\right|&\leq
\sin(|t_p|)\cos(t_p)\tan^2(\chi_p)
\leq\frac{\sin(2|t_p|)}{2\cdot 3}\\
\left|\tan(t_q)\left(\frac{g(\normal(q),\,\vartheta(q))}
         {g(\normal(q),\,\dot{\gamma}(q))}\right)^2\right|&\leq
\sin(|t_q|)\cos(t_q)\tan^2(\chi_q)
\leq\frac{\sin(2|t_q|)}{2\cdot 3}.
\end{align*}
Due to $\sin(2t_q)-\sin(2t_p) = 2\sin(t_q-t_p)\cos(t_p+t_q)$ and we arrive at
\begin{align}\label{eq:second bound}
    d''(s)&\leq
    -\frac{\II(p',\,p')}{\cos(\chi_p)}-\frac{\II(q',\,q')}{\cos(\chi_q)}
    -\sin(d)\cos(t_p+t_q)\left(1-\tfrac{2}{3}\right).
\end{align}
Now we claim that the second inequality of \eqref{eq:second derivatives} is implied by \eqref{eq:assumption 1} and the existence of $\delta$ which satisfy
\begin{align}\label{eq:assumption 2}
\delta&\leq\frac{1}{9}\sin\left(\tfrac{\lbcl}{2}\right)\cos(t_p+t_q),&
-\frac{\II(p',\,p')}{|p'|^2}&
\leq\frac{\delta}{2}\quad\mbox{ and }&
-\frac{\II(q',\,q')}{|q'|^2}&
\leq\frac{\delta}{2}.
\end{align}
The assumption \eqref{eq:assumption 1} implies $t_p\leq-d+\frac{\pi}{2}$ and $d-\frac{\pi}{2}\leq t_q$.

For $d\in[\frac{\pi}{2},\pi)$ we use $\cos(t_p)\leq\sin(d)$ and $\cos(t_q)\leq\sin(d)$ and for $d\in[\frac{\lbcl}{2},\frac{\pi}{2})$ we use $\cos(t_p)\leq 1$ and $\cos(t_q)\leq 1$ to obtain
\begin{align*}
\tfrac{1}{2\cdot 9}\sin(\tfrac{\lbcl}{2})
\max\left\{\tfrac{|p'|^2}{\cos(\chi_p)},\tfrac{|q'|^2}{\cos(\chi_q)}\right\}
\leq\tfrac{3\sqrt{3}}{8\cdot 9}\sin(\tfrac{\lbcl}{2})
\max\left\{\tfrac{|p'|^2}{\cos(\chi_p)},\tfrac{|q'|^2}{\cos(\chi_q)}\right\}
\leq\tfrac{1}{9}\sin(d).
\end{align*}
Together with \eqref{eq:assumption 2} we can estimate the first two terms in \eqref{eq:second bound}:
\begin{align*}
-\tfrac{\II(p',\,p')}{\cos(\chi_p)}\le
\tfrac{\delta}{2}\tfrac{|p'|^2}{\cos(\chi_p)}&\le
\tfrac{1}{9}\sin(d)\cos(t_p+t_q)&
-\tfrac{\II(q',\,q')}{\cos(\chi_q)}\le
\tfrac{\delta}{2}\tfrac{|q'|^2}{\cos(\chi_q)}&\le
\tfrac{1}{9}\sin(d)\cos(t_p+t_q).
\end{align*}
The third inequality of \eqref{eq:assumption 1} implies $\sin(\tfrac{\lbcl}{2})\cos(\tfrac{d}{2}) \le 2\sin(\tfrac{\lbcl}{4})\cos(\tfrac{d}{2})\le 2\sin(\tfrac{d}{2})\cos(\tfrac{d}{2})=\sin(d)$. Thus the second inequality of \eqref{eq:second derivatives} indeed follows with the help of \eqref{eq:second bound} from \eqref{eq:assumption 1} and \eqref{eq:assumption 2}.

We shall show first that there exists a vector field $\vartheta$ obeying at $s=0$
\begin{align*}
\delta&\leq\frac{1}{18}\sin\left(\tfrac{\lbcl}{2}\right)\cos(t_p+t_q),&
-\frac{\II(p',\,p')}{|p'|^2}&
\leq\frac{\delta}{4}\quad\mbox{ and }&
-\frac{\II(q',\,q')}{|q'|^2}&
\leq\frac{\delta}{4}.
\end{align*}
The Killing field $\vartheta$ is uniquely determined by two choices: firstly, the choice of a great 2-sphere $\Sp^2\subset\Sp^3$, which contains the closed geodesic from $p$ to $q$, and secondly, a choice of the zeroes of $\vartheta$, or equivalently a choice of the coordinates $t_p$ and $t_q$ with $t_q-t_p=d\mod \pi$. We start with $t_q=-t_p=\frac{d}{2}$ and set $\delta=\frac{1}{18}\sin(\frac{\lbcl}{2})$.

Now we choose the 2-sphere $\Sp^2\subset\Sp^3$ which contains the geodesic which connects $p$ and $q$.
This 2-sphere intersects $M$ along curves at $p$ and $q$. It is uniquely determined either by the line in $T_pM$ tangent to $\Sp^2$ or by a line in $T_qM$, which is tangent to $\Sp^2$. Since $\imm$ is a mean-convex Alexandrov embedding and both principal curvatures are uniformly bounded by $\kappa\ind{max}$, the cone angles of the double cones $\{X\in T_pM\mid \II(X,X)\geq-\frac{1}{4}\delta|X|^2\}$ and $\{X\in T_qM\mid \II(X,X)\geq-\frac{1}{4}\delta|X|^2\}$ are not smaller than $\tfrac{\pi}{2}+\Order(\delta)$. For sufficiently small $\epsilon\geq(\sin^2(\chi_p)+\sin^2(\chi_q))^{\frac{1}{2}}$ the tangent direction in the plane orthogonal to $\dot{\gamma}(p)$ in $T_pN$, and in the plane orthogonal to $\dot{\gamma}(q)$ in $T_qN$ of the corresponding spheres build two double cones with cone angles not smaller than $\tfrac{\pi}{2}$. Hence the intersection of both double cones is non-empty and there exists a 2-sphere where the intersecting curves of $\Sp^2 \cap M$ at $p$ and $q$ has tangent vectors which satisfy both
second and third conditions of $(\ref{eq:assumption 2})$.

Secondly we shall show that the inequalities \eqref{eq:assumption 1} and \eqref{eq:assumption 2} are satisfied for $s\in[0,\,s_0]$ with some $s_0>0$. Since the curvature is bounded by $\cot(\lbcl)$ and due to the assumption $(\sin^2(\chi_p)+\sin^2(\chi_q))^{\frac{1}{2}}\leq\epsilon$ there exists $s_0$ such that $t_p$ and $t_q$ do not reach the roots of $\vartheta$ for $s\in[0,\,s_0]$. Since the derivatives of $\cos(\chi_p)$, $\cos(\chi_q)$, $t_p$ and $t_q$ with respect to $s$ are uniformly bounded, there exists $s_0>0$ such that the inequalities \eqref{eq:assumption 1} and the first inequality of \eqref{eq:assumption 2} are satisfied for $s\in[0,\,s_0]$. Due to \eqref{eq:upper bound derivative of h} also the derivatives of $\II(p',p')$ and $\II(q',q')$ are uniformly bounded. Hence there exists $s_0>0$ only depending on $\lbcl$ and $\ubdh$, such that the second and the third inequality of \eqref{eq:second derivatives} are satisfied for $s\in[0,\,s_0]$. We choose $s_0>0$ small such that $|d(s)-d(s_0)| \leq c/2$.

Finally we have to satisfy the first inequality of \eqref{eq:second derivatives}. At the start point $s=0$ this is always the case for one choice of the sign of $\vartheta$. Now the second inequality of \eqref{eq:second derivatives} implies the first.
\end{proof}


\def\cydot{\leavevmode\raise.4ex\hbox{.}} \def\cprime{$'$}
\providecommand{\bysame}{\leavevmode\hbox to3em{\hrulefill}\thinspace}
\providecommand{\MR}{\relax\ifhmode\unskip\space\fi MR }
\providecommand{\MRhref}[2]{%
  \href{http://www.ams.org/mathscinet-getitem?mr=#1}{#2}
}
\providecommand{\href}[2]{#2}

\end{document}